\documentclass[11pt]{article}
\usepackage[utf8x]{inputenc}
\usepackage[T1]{fontenc}
 \usepackage{amsfonts}
 \usepackage{amsmath}

 \newcommand{\be}{\begin{equation}}
 \newcommand{\ee}{\end{equation}}
 \newcommand{\bea}{\begin{eqnarray}}
 \newcommand{\eea}{\end{eqnarray}}
 \newcommand{\beas}{\begin{eqnarray*}}
\newcommand{\eeas}{\end{eqnarray*}}
\newtheorem{theorem}{Theorem}[section]

\newtheorem{proposition}[theorem]{Proposition}
\newtheorem{corollary}[theorem]{Corollary}
\newtheorem{lemma}[theorem]{Lemma}
\newtheorem{remark}[theorem]{Remark}
\newtheorem{example}[theorem]{Example}
\newtheorem{examples}[theorem]{Examples}
\newtheorem{foo}[theorem]{Remarks}

\newenvironment{Examples}{\begin{examples}\rm}{\end{examples}}

\newenvironment{Remarks}{\begin{foo}\rm}{\end{foo}}
\newenvironment{proof}{\addvspace{\medskipamount}\par\noindent{\it Proof}.}
{\unskip\nobreak\hfill$\Box$\par\addvspace{\medskipamount}}

\newcommand{\ang}[1]{\left<#1\right>}  % angular brackets for projection
\newcommand{\brak}[1]{\left(#1\right)}    % round brackets
\newcommand{\crl}[1]{\left\{#1\right\}}   % curly brackets
\newcommand{\dom}{{\rm dom}\,}

\newcommand{\N}[1]{||#1||}     % Norm
\newcommand{\abs}[1]{\left|#1\right|}     % absolute value
\newcommand{\qed}{\unskip\nobreak\hfill$\Box$\par\addvspace{\medskipamount}}

\title{Representation of increasing convex functionals\\ 
with countably additive measures}

\author{Patrick Cheridito\thanks{Department of Mathematics, ETH Zurich, Switzerland.} 
\and Michael Kupper\thanks{Department of Mathematics, University of Konstanz, Germany.} 
\and Ludovic Tangpi\thanks{ORFE, Princeton University, USA.}}

\date{} %March 2021

\begin{document}
\maketitle

\begin{abstract}
We derive two types of representation results for increasing convex functionals in terms of countably additive 
measures. The first is a max-representation of functionals defined on spaces of real-valued continuous functions and 
the second a sup-representation of functionals defined on spaces of real-valued Borel measurable functions. Our 
assumptions consist of sequential semicontinuity conditions which are easy to verify in 
different applications.
\end{abstract}

\section{Introduction}
\label{sec:intro}

Let $\phi \colon X \to \mathbb{R} \cup \crl{+ \infty}$ be an increasing convex functional 
on a linear space $X$ of functions $f \colon \Omega \to \mathbb{R}$. More precisely, $\phi$ is convex and satisfies 
$\phi(f) \ge \phi(g)$ for $f \ge g$, where the second inequality is understood pointwise.
By $I(\phi)$ we denote the algebraic interior of the effective domain 
\[
\dom \phi := \crl{f \in X : \phi(f) < + \infty};
\] 
that is, $I(\phi)$ consists of all 
$f \in \dom \phi$ with the property that for every $g \in X$, 
there is an $\varepsilon > 0$ such that $f + \lambda g \in \dom \phi$ for all
$0 \le \lambda \le \varepsilon$. 

If $X$ is a linear space of  bounded measurable functions 
on a measurable space $(\Omega, {\cal F})$ containing all indicator functions $1_A$, 
$A \in {\cal F}$, it follows from standard convex duality arguments (see Section \ref{sec:max}) that 
\be \label{maxrep}
\phi(f) = \max_{\mu \in ba^+({\cal F})} 
\brak{\ang{f,\mu} - \phi^*_X(\mu)} \quad \mbox{for all } f \in I(\phi),
\ee
where $ba^+({\cal F})$ is the set of all finitely additive measures $\mu$ on ${\cal F}$ 
satisfying $\mu(\Omega) < +\infty$, $\ang{f,\mu}$ denotes the integral $\int_{\Omega} f d\mu$, 
and $\phi^*_X$ is the convex conjugate of $\phi$, given by
\be \label{defconj}
\phi^*_X(\mu) := \sup_{f \in X} \brak{\ang{f,\mu} - \phi(f)}.
\ee

In applications, a representation like \eqref{maxrep} is often more useful if it 
is in terms of countably instead of finitely additive measures. This paper provides  
such representations under different sequential semicontinuity conditions 
that are easy to verify in various concrete situations.

For a non-empty set $\Omega$, we call a linear subspace $X$ of $\mathbb{R}^{\Omega}$ a
{\sl Stone vector lattice} if for all $f,g \in X$, the point-wise minima $f \wedge g$ and $f \wedge 1$ 
also belong to $X$. By $\sigma(X)$ we denote the smallest $\sigma$-algebra on $\Omega$ 
making all functions $f \in X$ measurable with respect to the Borel $\sigma$-algebra on $\mathbb{R}$
and by $ca^+(X)$ all (countably additive) measures on $\sigma(X)$ satisfying $\ang{f,\mu} < + \infty$ for every 
$f \in X^+ := \crl{g \in X : g \ge 0}$.
$\phi^*_X \colon ca^+(X) \to \mathbb{R} \cup \crl{+ \infty}$ is defined as in \eqref{defconj}. 
If a sequence $(f_n)$ in $X$ converges pointwise from above to $f \in X$, we write $f_n \downarrow f$. 
Analogously, $f_n \uparrow f$ means pointwise convergence from below.

The following proposition is a non-linear extension of the Daniell--Stone theorem 
(see, e.g., Theorem 4.5.2 in \cite{Dudley}) 
and provides context to our main results, Theorems \ref{thmAB} and \ref{thm:Borel} below.
All proofs are given in Sections \ref{sec:max} and \ref{sec:sup}.

\begin{proposition} \label{prop:conv}
Let $\phi \colon X \to \mathbb{R} \cup \crl{+\infty}$ be an increasing convex functional on a 
Stone vector lattice $X$ over a non-empty set $\Omega$. Then
the implications {\rm (i)} $\Rightarrow$ {\rm (ii)} $\Rightarrow$ {\rm (iii)} $\Rightarrow$ {\rm (iv)}
$\Rightarrow$ {\rm (v)} $\Rightarrow$ {\rm (vi)} hold among the conditions:
\begin{itemize}
\item[{\rm (i)}]
There exists an $f \in I(\phi)$ such that $\phi(f_n) \downarrow \phi(f)$ for every sequence $(f_n)$ in 
$X$ satisfying $f_n \downarrow f$ 
\item[{\rm (ii)}]
$\phi(f_n) \downarrow \phi(f)$ for each $f \in I(\phi)$ and every 
sequence $(f_n)$ in $X$ satisfying $f_n \downarrow f$ 
\item[{\rm (iii)}]
For each $f \in I(\phi)$ and every sequence $(f_n)$ in $X^+$ satisfying $f_n \downarrow 0$
there exists a constant $\varepsilon > 0$ such that $\phi(f + \varepsilon f_n) \downarrow \phi(f)$
\item[{\rm (iv)}]
$\phi(f) = \max_{\mu \in ca^+(X)} (\ang{f,\mu} - \phi^*_{X}(\mu))$ for all $f \in I(\phi)$
\item[{\rm (v)}]
$\phi(f) = \sup_{\mu \in ca^+(X)} (\ang{f,\mu} - \phi^*_{X}(\mu))$ for all $f \in I(\phi)$
\item[{\rm (vi)}]
$\phi(f_n) \uparrow \phi(f)$ for each $f \in I(\phi)$ and every sequence $(f_n)$ in $X$ satisfying
$f_n \uparrow f$.
\end{itemize}
\end{proposition}
We are interested in representations of the form (iv) and (v). If $\phi$ is real-valued and linear, 
(i) is Daniell's condition \cite{Daniell} and equivalent to each of (ii), (iii) and (vi). 
However, for an increasing convex $\phi$, (i)--(iii) do not necessarily follow from (vi).
Also, in general (iii) is weaker than (ii), and there exist examples which do not satisfy any of the conditions (i)--(vi). 
These points are illustrated in the following 

\begin{Examples} \label{ex}
Consider the Stone vector lattice $l^{\infty}$ of all bounded functions $f \colon \mathbb{N} \to \mathbb{R}$,
where we use the convention $\mathbb{N} = \crl{1,2,...}$.
Denote by $ca^+_1(\mathbb{N})$ the set of all probability measures on $\mathbb{N}$ and 
by $ba^+_1(\mathbb{N})$ the set of all finitely additive probability measures on $\mathbb{N}$, that is,
all finitely additive measures $\mu$ on $\mathbb{N}$ satisfying $\mu(\mathbb{N}) =1$.\\[2mm]
{\bf 1.}
$s(f) := \sup_m f(m)$ defines an increasing convex functional $s \colon l^{\infty} \to \mathbb{R}$ which clearly fulfills (vi). 
It can easily be checked that the convex conjugate of $s$ is $s^*_{l^{\infty}}(\mu) = 0$ if $\mu$ belongs to 
$ba^+_1(\mathbb{N})$ and $s^*_{l^{\infty}}(\mu) = +\infty$ for all 
$\mu \in ba^+(\mathbb{N}) \setminus ba^+_1(\mathbb{N})$. One obviously has
\be \label{srep}
s(f) = \sup_{\mu \in ca^+_1(\mathbb{N})} \ang{f,\mu},
\ee
and it follows from \eqref{maxrep} that 
\be \label{smax}
s(f) = \max_{\mu \in ba^+_1(\mathbb{N})} \ang{f,\mu}.
\ee
\eqref{srep} is of the form (v). Moreover, for all $f \in l^{\infty}$ attaining their supremum, 
the supremum in \eqref{srep} is attained by a Dirac measure. But if $f \in l^{\infty}$ does not attain its
supremum, $s(f)$ cannot be written in the form (iv). So $s$ satisfies (v)--(vi) but
not (i)--(iv).\\[2mm]
{\bf 2.}
\be \label{repp}
p(f) = \sup_{\mu \in ca^+(\mathbb{N})} \ang{f,\mu}
\ee
defines an increasing convex functional $p \colon l^{\infty} \to \mathbb{R} \cup \crl{+ \infty}$ mapping $f$ to 
$0$ or $+ \infty$ depending on whether $s(f) \le 0$ or $s(f) > 0$. So 
$f$ belongs to $I(p)$ if and only if $s(f) < 0$, in which case the supremum in \eqref{repp} is attained.
It is easy to see that $p$ fulfills (iii) but not (ii). So by Proposition \ref{prop:conv}, it satisfies (iii)--(vi) but violates (i)--(ii).\\[2mm]
{\bf 3.}
Now pick an increasing $f \in l^{\infty}$ that does not attain its supremum, and choose a $\mu \in ba^+_1(\mathbb{N})$
which maximizes \eqref{smax}. Then one has for all $n$,
\[
s(f) = \ang{f 1_{[1,n]}, \mu} + \ang{f1_{(n,\infty)},\mu} \le f(n) \mu[1,n] + s(f) (1-\mu[1,n]).
\]
It follows that $\mu[1,n] = 0$ for all $n$. So the positive linear functional $l \colon l^{\infty} \to \mathbb{R}$, given by
$l(f) := \ang{f,\mu}$, satisfies $l(1_{[1,n]}) = 0 < l(1) = 1$, showing that it violates condition (vi), and 
therefore also (i)--(v).
\end{Examples}

In the following we introduce four conditions, called (A), (B), (C) and (D), for an increasing convex 
functional $\phi \colon X \to \mathbb{R} \cup \crl{+\infty}$ on a Stone vector lattice $X$ of functions 
$f \colon \Omega \to \mathbb{R}$ defined on a topological space $\Omega$. 
(A) and (B) are sequential-continuity-from-above conditions, which both imply a max-representation 
like (iv) of Proposition \ref{prop:conv} in the case where $X$ is a stone vector lattice of continuous functions. 
Similarly, (C) and (D) are sequential-continuity-from-below conditions which allow us to derive a sup-representation 
similar to (v) of Proposition \ref{prop:conv} if $X$ is the set of all bounded Borel measurable functions on a 
Hausdorff space $\Omega$.

\begin{itemize}
\item[{\bf (A)}] 
For all $f \in I(\phi)$ and every sequence $(f_n)$ in $X^+$ satisfying $f_n \downarrow 0$,
there exists a constant $\varepsilon > 0$ such that for each constant $\delta > 0$, there are 
$m \in \mathbb{N}$, $g \in \mathbb{R}_+^{\Omega}$ and an increasing convex function
$\hat{\phi} \colon  Y \to \mathbb{R}$ on a convex subset $Y \subseteq \mathbb{R}_+^{\Omega}$ 
containing $\crl{0, f_m g, (\varepsilon-g)^+, \varepsilon f_n : n \ge m}$ so that 
\begin{itemize}
\item[{\rm (i)}]
$\crl{g < \varepsilon}$ is relatively compact,
\item[{\rm (ii)}]
$\hat{\phi}(f_mg) \le \delta$ and
\item[{\rm (iii)}]
$\hat{\phi}(\varepsilon f_n) \ge \phi(f+ \varepsilon f_n) - \phi(f)$ for all $n \ge m$.
\end{itemize}

\item[{\bf (B)}]
For all $f \in I(\phi)$ and every sequence $(f_n)$ in $X^+$ satisfying $f_n \downarrow 0$,
there exist functions $g,g_1,g_2, \dots$ in $\mathbb{R}_+^{\Omega}$ and numbers
$m, m_1, m_2, \dots$ in $\mathbb{N}$
together with an increasing convex function $\hat{\phi} \colon Y \to \mathbb{R}$ on a convex subset
$Y \subseteq \mathbb{R}_+^{\Omega}$ containing $\crl{0, f_n/m, g, g_n : n \ge m}$ so that 
\begin{itemize}
\item[{\rm (i)}]
$\crl{f_m >g/n}$ is relatively compact and contained in $\crl{m_n g_n \ge 1}$ for all $n \ge m$,
\item[{\rm (ii)}]
$\hat{\phi}(0) = 0$ and
\item[{\rm (iii)}]
$\hat{\phi}(f_n/m) \ge \phi(f+f_n/m) - \phi(f)$ for all $n \ge m$.
\end{itemize}
\end{itemize}

\begin{theorem} \label{thmAB}
Let $X$ be a Stone vector lattice $X$ of continuous functions $f \colon \Omega \to \mathbb{R}$ on 
a topological space $\Omega$ and $\phi \colon X \to \mathbb{R} \cup \crl{+ \infty}$ an increasing 
convex functional satisfying {\rm (A)} or {\rm (B)}. Then
\[
\phi(f) = \max_{\mu \in ca^+(X)} \brak{\ang{f,\mu} - \phi^*_X(\mu)} \quad 
\mbox{for all } f \in I(\phi).
\]
\end{theorem}

As a special case of Theorem \ref{thmAB}, one obtains the following variant of the 
Daniell--Stone theorem:

\begin{corollary} \label{corAB}
If $X$ is a Stone vector lattice of continuous functions $f \colon \Omega \to \mathbb{R}$ on a topological space $\Omega$, 
then every positive linear functional $\phi \colon X \to \mathbb{R}$ satisfying {\rm (A)} or
{\rm (B)} is of the form $\phi(f) = \ang{f,\mu}$, $f \in X$, for a measure $\mu \in ca^+(X)$.
\end{corollary}

In various situations, a measure on a $\sigma$-algebra ${\cal F}$ of subsets of a topological space 
$\Omega$ can be shown to possess regularity properties. Let us call a finite measure $\mu$ on ${\cal F}$
{\sl closed regular} if 
\[
\mu(A) = \sup \crl{\mu(B) : \mbox{$B \in {\cal F}$, $B$ is closed and $B \subseteq A$}}
\quad \mbox{for all } A \in {\cal F}
\]
and {\sl regular} if 
\[
\mu(A) = \sup \crl{\mu(B) : \mbox{$B \in {\cal F}$, $B$ is closed, compact and $B \subseteq A$}}
\]
for all $A \in {\cal F}.$
If $X$ is a Stone vector lattice of real-valued functions containing the constant functions, then every measure 
$\mu \in ca^+(X)$ is finite. Moreover, standard arguments (see Section \ref{sec:max} for details) yield the following
result:

\begin{proposition} \label{prop:reg}
Let $X$ be a family of continuous functions $f \colon \Omega \to \mathbb{R}$ on a topological space
$\Omega$. Then every finite measure $\mu$ on $\sigma(X)$ is closed regular.
Furthermore, if $\mu$ is a finite measure on $\sigma(X)$ and there exists a sequence $(K_n)$ 
of compact sets in $\sigma(X)$ such that $\mu(K_n) \to \mu(\Omega)$, then $\mu$ is regular. 
\end{proposition}

For regularity results for measures representing continuous linear functionals on 
$C_b$ in different settings, we refer to Section 7.10 of \cite{Bog}.

\begin{Examples} $\mbox{}$\\
{\bf 1. (Tightness conditions)}\\
Let $\phi \colon C_b \to \mathbb{R} \cup \crl{+ \infty}$ be an increasing convex functional 
on the set $C_b$ of all bounded continuous functions  $f \colon \Omega \to \mathbb{R}$
on a topological space $\Omega$. Assume $V$ is a linear space containing all functions of the form 
$f1_K$ and $f 1_{K^c}$ for $f \in C_b$ and $K$ a compact subset of $\Omega$. If $\phi$ can be extended
to an increasing convex $\psi \colon V \to \mathbb{R} \cup \crl{+ \infty}$ 
with the property that for every $f \in I(\phi)$, there 
exists a constant $\delta > 0$ and a sequence $(K_n)$ of compact sets such that 
\be \label{tight1}
\psi(f + \delta 1_{K^c_n}) \to \psi(f),
\ee then 
for every $f \in I(\phi)$ and each sequence $(f_n) \in C^+_b$ satisfying $f_n \downarrow 0$, 
there exists a constant $\varepsilon > 0$ such that $\phi(f + \varepsilon) < + \infty$ and
\[
\psi(f + \varepsilon f_1 1_{K^c_n}) - \psi(f) \le 
\psi(f + \varepsilon \N{f_1}_{\infty} 1_{K^c_n}) - \psi(f) \to 0 \quad 
\mbox{as } n \to \infty.
\]
It follows that condition (A) holds with $\hat{\phi}(h) = \phi(f+h) - \phi(f)$, 
and one obtains from Theorem \ref{thmAB} that
\be \label{tightrep}
\phi(f) = \max_{\mu \in ca^+(C_b)} (\ang{f,\mu} - \phi^*_{C_b}(\mu)) \quad 
\mbox{for all } f \in I(\phi).
\ee

If $\Omega$ is Hausdorff, then all compact sets $K \subseteq \Omega$ 
are closed and therefore, belong to the Borel $\sigma$-algebra ${\cal F}$. So in this case, if
$\phi \colon B_b \to \mathbb{R}$ is an increasing convex functional defined on the space $B_b$ 
of all bounded Borel measurable functions $f \colon \Omega \to \mathbb{R}$ with the property that for every 
constant $M \ge 1$, there exists a sequence $(K_n)$ of compact sets such that 
\be \label{tight2}
\phi(M 1_{K^c_n}) \to \phi(0),
\ee
one deduces as in the proof of (i) $\Rightarrow$ (ii) of Proposition \ref{prop:conv} that $\phi$
satisfies condition \eqref{tight1} for all $f \in C_b$, and as a consequence, 
is representable as \eqref{tightrep} on $C_b$.
If in addition, $\phi$ has the translation property:
$\phi(f + m) = \phi(f) + m$ for all $f \in B_b$ and $m \in \mathbb{R}$, then \eqref{tight2}
holds if and only if for every $M \ge 1$, there exists a sequence of compacts $(K_n)$ such that 
\be \label{tight3}
\phi(-M 1_{K_n}) \to \phi(-M),
\ee
which is the tightness condition used in Proposition 4.30 of \cite{FS16} to 
derive a max-representation for convex risk measures. Note that if $\phi$ has the 
translation property, then $\phi^*_{C_b}(\mu) = + \infty$ for $\mu \in ca^+(C_b) \setminus ca^+_1(C_b)$,
and consequently, \eqref{tightrep} reduces to a maximum over probability measures:
\be \label{rm}
\phi(f) = \max_{\mu \in ca^+_1(C_b)} (\ang{f,\mu} - \phi^*(\mu)), \quad f \in C_b.
\ee

In the special case where $\Omega$ is a metric space, $C_b$ generates the 
Borel $\sigma$-algebra ${\cal F}$, and for every compact set $K_n$, there 
exists a sequence $(h^n_m)$ of functions in $C_b^+$ such that $h^n_m \uparrow 1_{K_n^c}$.
Therefore, if $\phi \colon C_b \to \mathbb{R} \cup \crl{+\infty}$ is an increasing convex functional 
with an increasing convex extension $\psi$ satisfying \eqref{tight1}, then 
for any $f \in I(\phi)$ and $\mu \in ca^+({\cal F})$ maximizing \eqref{tightrep}, there exists a 
constant $\delta > 0$ and a sequence $(K_n)$ of compact sets such that
\beas
\delta \ang{1_{K^c_n},\mu} &=& \lim_{m \to + \infty} \delta \ang{h^n_m,\mu}
\le \lim_{m \to + \infty} \phi(f + \delta h^n_m) - \phi(f)\\ &\le& \psi(f + \delta 1_{K^c_n}) - \psi(f)) \downarrow 0
\quad \mbox{as } n \to + \infty.
\eeas
So it follows from Proposition \ref{prop:reg} that $\mu$ is regular, and as a result, the representations 
\eqref{tightrep} and \eqref{rm} can be written as maxima over regular finite measures or regular 
probability measures on ${\cal F}$, respectively.\\[2mm]
{\bf 2. (Adapted spaces and cones)}\\
Let $\psi \colon V \to \mathbb{R} \cup \crl{+\infty}$ be an increasing convex functional, 
where $V$ is an adapted space \cite{C} or an adapted cone \cite{MS} of continuous functions 
$f \colon \Omega \to \mathbb{R}$ on a topological space $\Omega$.
That is, $V$ is either a linear space satisfying 
\begin{itemize}
\item[{\rm (i)}] $V = V^+ - V^+$, where $V^+ = \crl{f \in V : f \ge 0}$;
\item[{\rm (ii)}] for every $\omega \in \Omega$, there exists an $f \in V^+$ such that $f(\omega) > 0$ and
\item[{\rm (iii)}] for every $f \in V^+$, there exists a $g \in V^+$ such that  for every constant $\varepsilon > 0$,
the set $\crl{f > \varepsilon g}$ is relatively compact, 
\end{itemize}
or $V$ is a convex cone with the properties 
\begin{itemize}
\item[{\rm (i)}] 
$V = V^+ \cup \crl{0}$;
\item[{\rm (ii)}]
for every $\omega \in \Omega$, there exists an $f \in V$ such that $f(\omega) > 0$ and
\item[{\rm (iii)}] for every $f \in V$, there exists a $g \in V$ such that  for every constant $\varepsilon > 0$,
the set $\crl{f > \varepsilon g}$ is relatively compact.
\end{itemize}
In both cases,
\[
X := \crl{f : \Omega \to \mathbb{R} \mbox{ continuous }: |f| \le g \mbox{ for some } g \in V}
\]
is a Stone vector lattice containing $V$, and 
\[
\phi(f) := \inf \crl{\psi(g) : f \le g, \, g \in V}
\]
defines an increasing convex extension $\phi \colon X \to \mathbb{R} \cup \crl{+ \infty}$ of $\psi$. 
Furthermore, for $f \in I(\phi)$ and a sequence $(f_n)$ in $X^+$ satisfying $f_n \downarrow 0$,
there is a constant $\varepsilon >0$ such that $\phi(f + \varepsilon f_1) < + \infty$. It follows from (iii) that
there exists a $g \in V^+$ such that $\phi(f+g) < + \infty$ and the set $\crl{f_1 > g/n}$ 
is relatively compact for all $n \in \mathbb{N}$. 
By compactness, one obtains from (ii) that there exist functions $g_n \in V^+$ and 
numbers $m_n \in \mathbb{N}$, $n \in \mathbb{N}$, such that $\phi(f+g_n) < + \infty$ and 
$m_n g_n \ge 1$ on $\crl{f_1 > g/n}$. This shows that condition (B) holds 
with $\hat{\phi}(h) = \phi(f+h) - \phi(f)$.
So by Theorem \ref{thmAB}, 
\be \label{phimax}
\phi(f) = \max_{\mu \in ca^+(X)} (\ang{f,\mu} - \phi^*_X(\mu)) \quad \mbox{for all } f \in I(\phi).
\ee
Moreover, it follows from the definition of $\phi$ that $I(\psi) \subseteq I(\phi)$ and $\psi^*_V(\mu) = \phi^*_X(\mu)$ for 
$\mu \in ca^+(X)$. Therefore,
\be \label{psimax}
\psi(f) = \max_{\mu \in ca^+(X)} (\ang{f,\mu} - \psi_V^*(\mu)), \quad \mbox{for all } f \in I(\psi).
\ee
\eqref{phimax} and \eqref{psimax} are non-linear versions of the linear representation results,
Proposition 2 in \cite{C} and Proposition 11 in \cite{MS}. But in contrast to \cite{C,MS}, here $X$
does not have to be locally compact. 
\end{Examples}

The next result gives a sup-representation for increasing convex functionals $\phi$ on the space 
$B_b$ of all bounded Borel measurable functions $f \colon \Omega \to \mathbb{R}$ on a Hausdorff space 
$\Omega$. The following two conditions are variants of (vi) in Proposition \ref{prop:conv}. 
We call a sequence $(K_n)$ of subsets of $\Omega$ or a sequence 
$(f_n)$ of real-valued functions on $\Omega$ {\sl increasing} if $K_n \subseteq K_{n+1}$ or $f_n \le f_{n+1}$ for all $n$,
respectively.

\begin{itemize}
\item[{\bf (C)}]
$\phi$ is real-valued and there
exists an increasing sequence $(K_n)$ of compact subsets of $\Omega$ such that 
$\phi(f_n) \uparrow \phi(f)$ for every increasing sequence $(f_n)$ in $B_b$ and $f \in B_b$
such that $|f-f_n| 1_{K_m}= 0$ for all $n \ge m$.
\item[{\bf (D)}]
There exists an increasing sequence $(K_n)$ of compact subsets of $\Omega$ such that 
$\phi(f_n) \uparrow \phi(f)$ for every increasing sequence $(f_n)$ in $B_b$ and $f \in B_b$
such that $|f-f_n| 1_{K_m} \le 1/m$ for all $n \ge m$.
\end{itemize}

Let us denote by $U_b$ the set of all bounded upper semicontinuous functions $f \colon \Omega \to \mathbb{R}$.
We define the lower regularization of $\phi$ by
\[
\phi_r(f) := \sup \crl{\phi(g) : g \in U_b , \, g \le f}
\]
and say $\phi$ is {\sl lower regular} if $\phi = \phi_r$. By ${\cal F}$ we denote the 
Borel $\sigma$-algebra on $\Omega$ and by $ca^+_r({\cal F})$ the collection of all regular 
finite measures on ${\cal F}$. For 
\[
\phi^*_{C_b}(\mu) := \sup_{f \in C_b} \brak{\ang{f,\mu} - \phi(f)} \quad \mbox{and} \quad 
\phi^*_{U_b}(\mu) := \sup_{f \in U_b} \brak{\ang{f,\mu} - \phi(f)},
\]
one obviously has $\phi^*_{C_b}(\mu) \le \phi^*_{U_b}(\mu)$ for all $\mu \in ca^+_r({\cal F})$.

\begin{theorem} \label{thm:Borel}
Let $\Omega$ be a Hausdorff space with Borel $\sigma$-algebra ${\cal F}$ and
$\phi \colon B_b \to \mathbb{R} \cup \crl{+ \infty}$ an increasing convex functional.
If $\phi$ satisfies {\rm (C)} or {\rm (D)}, then 
\bea \label{repcont}
\phi(f) = \sup_{\mu \in ca^+_r({\cal F})} \brak{\ang{f,\mu} - \phi^*_{C_b}(\mu)} \quad \mbox{for all } f \in C_b,\\
\label{inequsc}
\phi(f) \le \sup_{\mu \in ca^+_r({\cal F})} \brak{\ang{f,\mu} - \phi^*_{C_b}(\mu)} \quad
\mbox{for all } f \in U_b,\\
\label{ineqBorel}
\phi_r(f) \le \sup_{\mu \in ca^+_r({\cal F})} \brak{\ang{f,\mu} - \phi^*_{C_b}(\mu)} \mbox{ for all } f \in B_b,
\eea
and both inequalities become equalities if $\phi^*_{C_b}(\mu) = \phi^*_{U_b}(\mu)$ for all $\mu \in ca^+_r({\cal F})$.
In particular, if $\phi$ is lower regular and $\phi^*_{C_b}(\mu) = \phi^*_{U_b}(\mu)$ 
for all $\mu \in ca^+_r({\cal F})$, then 
\be \label{repBorel}
\phi(f) = \sup_{\mu \in ca^+_r({\cal F})} \brak{\ang{f,\mu} - \phi^*_{C_b}(\mu)}
\quad \mbox{for all } f \in B_b.
\ee
\end{theorem}

For positive linear functionals, Theorem \ref{thm:Borel} yields the following:

\begin{corollary} \label{cor:lin}
Let $\Omega$ be a Hausdorff space with Borel $\sigma$-algebra ${\cal F}$ and 
$\phi \colon B_b \to \mathbb{R}$ a positive linear functional satisfying {\rm (C)}. 
Then there exists a $\mu \in ca^+_r({\cal F})$ such that 
\be \label{lincont}
\phi(f) = \ang{f,\mu} \quad \mbox{for all } f \in C_b.
\ee
If $\Omega$ is a metric space with Borel $\sigma$-algebra ${\cal F}$, one also has
\be \label{linineq}
\phi(f) \le \ang{f,\mu} \mbox{ for all } f \in U_b \quad \mbox{and} \quad
\phi_r(f) \le \ang{f,\mu} \mbox{ for all } f \in B_b,
\ee
where the inequalities are equalities if $\phi^*_{C_b}(\nu) = \phi^*_{U_b}(\nu)$ 
for all $\nu \in ca^+_r({\cal F})$.
In particular, if $\Omega$ is a metric space with Borel $\sigma$-algebra ${\cal F}$, 
$\phi$ is lower regular, and $\phi^*_{C_b}(\nu) = \phi^*_{U_b}(\nu)$ for all $\nu \in ca^+_r({\cal F})$,
then 
\be \label{linBorel}
\phi(f) = \ang{f,\mu} \quad \mbox{for all } f \in B_b.
\ee
\end{corollary}

\begin{Remarks} \label{remlowerreg} $\mbox{}$\\
{\bf 1.}
To have a representation of the form \eqref{repBorel} or \eqref{linBorel}, it is necessary that $\phi$ be
lower regular. Indeed, for every $f \in B_b$ and constant $\delta > 0$, there exists a measurable 
partition $(A_m)_{m=1}^M$ of $\Omega$ and numbers $a_1 < \dots < a_M$ such that the step function
$g = \sum_{m=1}^M a_m 1_{A_m}$ satisfies $g \le f \le g + \delta$. Furthermore, for
each $\mu \in ca^+_r({\cal F})$, one can choose closed sets $F_m \subseteq A_m$ 
such that $\ang{g,\mu} \le \ang{h,\mu} + \delta$ for the upper semicontinuous function
$h = a_1 1_{\Omega \setminus \bigcup_{m=2}^M F_m} + \sum_{m =2}^M a_m 1_{F_m} \le g$.
It follows that $\ang{f,\mu} \le \ang{h,\mu} + \delta (\ang{1,\mu}+1)$. So 
any linear functional of the form \eqref{linBorel} is lower regular, and as a supremum of 
lower regular functionals, \eqref{repBorel} is again lower regular.\\[2mm]
{\bf 2.}
If $\Omega$ is a Hausdorff space with Borel $\sigma$-algebra ${\cal F}$, it
follows from 1. that for all $\mu \in ca^+_r({\cal F})$ and $f \in B_b$,
there exists a sequence $(f_n)$ in $U_b$ such that $f_n \le f$ and $\ang{f_n, \mu} \uparrow \ang{f,\mu}$.
As a result, one obtains for every increasing functional 
$\phi \colon B_b \to \mathbb{R} \cup \crl{+\infty}$ and $\mu \in ca^+_r({\cal F})$,
\[
\phi^*_{U_b}(\mu) = \sup_{f \in U_b} \brak{\ang{f,\mu} - \phi(f)} 
= \phi^*_{B_b}(\mu) = \sup_{f \in B_b} \brak{\ang{f,\mu} - \phi(f)}.
\]
Similarly, if $\mu \in ca^+_r({\cal F})$ has the 
property that for all $f \in U_b$, there exists a sequence $(f_n)$ in $C_b$ such that $f_n \le f$ and 
$\ang{f_n,\mu} \uparrow \ang{f,\mu}$, then $\phi^*_{C_b}(\mu) = \phi^*_{U_b}(\mu)$ for 
every increasing functional $\phi \colon B_b \to \mathbb{R} \cup \crl{+\infty}$. This
provides a sufficient condition for the inequalities in \eqref{inequsc}, \eqref{ineqBorel} 
and \eqref{linineq} to be equalities. 
\end{Remarks}

Further related literature includes \cite{DS94, DS98}, which established the equivalence
of the sequential continuity condition {\sl no free lunch with vanishing risk} to 
the existence of a countably additive pricing measure in general financial market models.
In the context of financial risk measures, the relation of sequential semicontinuity conditions of the form (ii) 
and (vi) of Proposition \ref{prop:conv} to the representability with countably additive measures
has been studied in \cite{Del2000, Del2002, FS02a, FS02b, Kr, FS16}. Tightness conditions 
for financial risk measures similar to \eqref{tight1}, \eqref{tight2} and \eqref{tight3} have been 
investigated in \cite{FS02b, Kr, FS16, BDT}. Related tightness arguments also play a role in the 
derivation of representation results in \cite{CKT, BCK, BDT, CKPS}. Sequential continuity conditions 
ensuring the representability of preference functionals in terms of countably additive measures have 
been used in \cite{MMR, CMMM, DK}.
Max-representations for increasing convex functionals in terms of 
topological dual elements can be deduced from automatic continuity results such as, e.g.,
\cite{Bor, RS, CL, BF}. If $X$ is an $L^p$-space for some $p \in [1,+\infty)$, or more generally, 
an Orlicz heart, this yields max-representations with countably additive measures.
But in general, not all topological dual elements correspond to
countably additive measures. For instance, if $X=C_b$ or $B_b$, 
the supremum in (v) of Proposition \ref{prop:conv} is not always attained by a countably additive measure. 
Conditions (C) and (D) are related to order lower semicontinuity; see, e.g., Lemma 3 of \cite{BF}. 
However, since they are sequential, they typically are easier to verify in applications.
%Notice that in general (e.g. for $X=C_b$), the condition (ii) in Proposition \ref{prop:conv} is weaker than the order upper semicontinuity. Moreover, dual representations of proper convex functionals can also be derived from weak semicontinuity. However, the sequential continuity conditions proposed in the present paper are often easier to check in applications, see e.g., \cite{CKT,BCK,BDT}. 
	
The remainder of the paper is organized as follows: In Section \ref{sec:max} we prove 
representation \eqref{maxrep}, Proposition \ref{prop:conv}, Theorem \ref{thmAB}, Corollary \ref{corAB}
and Proposition \ref{prop:reg}.
In Section \ref{sec:sup} we give the proofs of Theorem \ref{thm:Borel} and Corollary \ref{cor:lin}.

\setcounter{equation}{0}
\section{Derivation of max-representations}
\label{sec:max}

{\bf Proof of representation \eqref{maxrep}}\\
It is immediate from the definition of $\phi^*_X$ that 
\be \label{genineq}
\phi(f) \ge \sup_{\mu \in ba^+({\cal F})} \brak{\ang{f,\mu} - \phi^*_X(\mu)} 
\quad \mbox{for every } f \in X.
\ee
On the other hand, for $f \in I(\phi)$, the directional derivative 
\[
\phi'(f;g) := \lim_{\varepsilon \downarrow 0} \frac{\phi(f + \varepsilon g) - \phi(f)}{\varepsilon}
\]
is a real-valued increasing sublinear function of $g \in X$. So it follows from the 
Hahn--Banach extension theorem that there exists a linear functional $\psi \colon X \to \mathbb{R}$ satisfying
\be \label{HB}
\psi(g) \le \phi'(f;g) \le \phi(f+g) - \phi(f)
\ee
for all $g \in X$. In particular, 
\be \label{mon}
\psi(- g) \le \phi(f - g) - \phi(f) \le 0 \quad \mbox{for all } g \in X^+, 
\ee
which shows that $\nu(A) := \psi(1_A)$, $A \in {\cal F}$, defines a 
finitely additive measure $\nu \in ba^+({\cal F})$. Moreover, since $\psi(\lambda 1) = \lambda \psi(1)$, 
$\lambda \in \mathbb{R}$, one obtains from \eqref{mon} that $\psi$ is continuous with 
respect to the sup-norm on $X$, and it follows that 
$\psi(g) = \ang{g,\nu}$ for all $g \in X$; see, e.g., Theorem IV.5.1 in \cite{DS}. \eqref{HB} also implies 
\[
\phi^*_X(\nu) \le \ang{f, \nu} - \phi(f),
\]
with together with \eqref{genineq}, gives 
$\phi(f) = \max_{\mu \in ba^+({\cal F})} \brak{\ang{f,\mu} - \phi^*_X(\mu)}$.
\qed

\bigskip
\noindent
{\bf Proof of Proposition \ref{prop:conv}}\\
To prove (i) $\Rightarrow$ (ii), consider $f,g \in I(\phi)$ such that $f$ fulfills (i). Then for 
all $\lambda \in (0,1)$ and every sequence $(f_n)$ in $X^+$ satisfying $f_n \downarrow 0$, 
one has
\beas
&& \phi(g + f_n) \le \lambda \phi \brak{f + \frac{1}{\lambda} f_n}
+ (1-\lambda) \phi \brak{\frac{g - \lambda f}{1-\lambda}}\\
&=& \lambda \phi \brak{f + \frac{1}{\lambda} f_n}
+ (1-\lambda) \phi \brak{g + \frac{\lambda}{1-\lambda} (g-f)}.
\eeas
Since $f$ satisfies (i), 
\[
\phi \brak{f + \frac{1}{\lambda} f_n} \downarrow \phi(f) \quad \mbox{for fixed } 
\lambda \in (0,1) \mbox{ and } n \to +\infty.
\]
Moreover, there exists a constant $\delta > 0$ such that $x \mapsto \phi(g + x(g-f))$ is 
a real-valued convex function on the interval $(-\delta, \delta)$.
As a consequence, it is continuous at $0$, and one obtains
\[
\lambda \phi(f) +  (1-\lambda) \phi \brak{g + \frac{\lambda}{1-\lambda} (g-f)} \to \phi(g)
\quad \mbox{for } \lambda \downarrow 0.
\]
This shows that $\phi(g + f_n) \downarrow \phi(g)$.

(ii) $\Rightarrow$ (iii) is obvious.
To prove (iii) $\Rightarrow$ (iv), note first that it follows from the definition of 
$\phi^*_X$ that 
\[
\phi(f) \ge \sup_{\mu \in ca^+(X)} (\ang{f,\mu} - \phi^*_X(\mu)) \quad \mbox{for all } f \in X.
\]
Moreover, for $f \in I(\phi)$, one deduces as in the proof of representation \eqref{maxrep} 
from the Hahn--Banach extension theorem that
there exists a positive linear functional $\psi \colon X \to \mathbb{R}$ satisfying
\[
\psi(g) \le \phi'(f;g) \le \phi(f+g) - \phi(f), \quad g \in X. 
\]
If (iii) holds, then for every sequence $(f_n)$ in $X^+$ satisfying $f_n \downarrow 0$, there exists 
a constant $\varepsilon > 0$ such that
\[
\varepsilon \psi(f_n) \le \phi(f + \varepsilon f_n) - \phi(f) \downarrow 0.
\]
So one obtains from the Daniell--Stone theorem a $\nu \in ca^+(X)$ such that 
$\psi(g) = \ang{g,\nu}$ for all $g \in X$. It follows that $\phi(f) + \phi^*_X(\nu) = \ang{f,\nu}$,
which implies that $\phi(f) = \max_{ca^+(X)} \brak{\ang{f,\mu} - \phi^*_X(\mu)}$.

(iv) $\Rightarrow$ (v) is clear, and (v) $\Rightarrow$ (vi) follows since by 
the monotone convergence theorem, the mapping $f \mapsto \ang{f,\mu} - \phi^*_X(\mu)$ 
satisfies (v) for every $\mu \in ca^+(X)$.
\qed

\bigskip
\noindent
{\bf Proof of Theorem \ref{thmAB}}\\
Choose a function $f \in I(\phi)$ and a 
sequence $(f_n)$ in $X^+$ satisfying $f_n \downarrow 0$. If we can show that there exists a 
constant $\varepsilon > 0$ such that $\phi(f+ \varepsilon f_n) \downarrow \phi(f)$, 
the theorem follows from Proposition \ref{prop:conv}. 
Let us first assume $\phi$ satisfies (A). Then there exists a constant $\lambda > 0$ such that for 
every constant $\delta > 0$, there are $m \in \mathbb{N}$, $g \in \mathbb{R}_+^{\Omega}$ and an 
increasing convex function $\hat{\phi} \colon  Y \to \mathbb{R}$ on a convex subset 
$Y \subseteq \mathbb{R}_+^{\Omega}$ containing 
$\crl{0, f_m g, (\lambda-g)^+, \lambda f_n : n \ge m}$ such that 
$\crl{g < \lambda}$ is relatively compact, $\hat{\phi}(f_m g) \le \delta$, and 
$\hat{\phi}(\lambda f_n) \ge \phi(f+ \lambda f_n) - \phi(f)$ for all $n \ge m$.
Since $x \mapsto \hat{\phi}(x (\lambda-g)^+)$ is a real-valued increasing convex function on the interval 
$[0,1]$, it must be continuous at $0$. In particular, there exists an $x \in (0,1]$ such that 
\[
\hat{\phi} (x (\lambda-g)^+) \le \hat{\phi}(0) + \delta \le \hat{\phi}(f_m g) + \delta \le 2 \delta.
\]
For $n \ge m$, one has $\lambda f_n \le f_m g + f_n(\lambda -g)^+$, and by
Dini's lemma, $f_n$ converges to $0$ uniformly on the closure of $\crl{g < \lambda}$.
So there exists an $n \ge m$ such that 
\[
f_n(\lambda -g)^+ \le x (\lambda-g)^+,
\]
and one obtains
\begin{align*}
 \phi \brak{f+ \frac{\lambda}{2} f_n} - \phi(f) 
&\le \hat{\phi}\brak{\frac{\lambda}{2} f_n}  \le \hat{\phi} \brak{\frac{f_m g +x (\lambda-g)^+}{2}}\\
&\le \frac{\hat{\phi}(f_mg) + \hat{\phi}(x(\lambda -g)^+)}{2} \le 2 \delta.
\end{align*}
Since $\delta > 0$ was arbitrary, this shows that $\phi(f + \lambda f_n/2) \downarrow \phi(f)$. 

If $f$ satisfies condition (B), there exist functions $g,g_1,g_2, \dots$ in $\mathbb{R}_+^{\Omega}$ 
and numbers $m,m_1,m_2, ...$ in $\mathbb{N}$ together with an increasing convex function 
$\hat{\phi} \colon Y \to \mathbb{R}$ on a convex subset $Y \subseteq \mathbb{R}_+^{\Omega}$ 
containing the set \linebreak
$\crl{0, f_n/m, g, g_n : n \ge m}$ such that $\crl{f_m > g/n}$ is relatively compact and 
contained in $\crl{m_n g_n \ge 1}$ for all $n \ge m$,
$\hat{\phi}(0) = 0$, and $\hat{\phi}(f_n/m) \le \phi(f + f_n/m) - \phi(f)$ for all 
$n \ge m$. Since $x \mapsto \hat{\phi}(x g)$ is a
real-valued increasing convex function on the interval $[0,1]$, it is continuous at $0$. 
In particular, for a given constant $\delta > 0$, there exists an integer $k \ge 2m$ such that 
$\hat{\phi}(2g/km) \le \delta$. Similarly, there exists an integer $l \ge 2 m_k$ such 
that $\hat{\phi}(2m_kg_k/lm) \le \delta$. By Dini's Lemma, $f_n$ converges uniformly to 
$0$ on the closure of the set $\crl{f_m > g/k}$.
So there exists an $n \ge m$ such that $f_n \le 1/l$ on $\crl{f_m > g/k}$.
Since $\crl{f_m > g/k}$ is contained in $\crl{m_k g_k \ge 1}$ and $f_n \le f_m \le g/k$ on 
$\crl{f_m \le g/k}$, one has $(f_n- g/k)^+ \le m_k g_k/l$. Therefore,
\[
f_n \le g/k + (f_n-g/k)^+ \le g/k + m_k g_k/l,
\]
and 
\begin{align*}
	\phi \brak{f+ \frac{f_n}{m}} - \phi(f) &\le \hat{\phi} \brak{\frac{f_n}{m}} \le \hat{\phi} \brak{\frac{g}{km} + \frac{m_k g_k}{lm}}\\
	&
\le \frac{\hat{\phi}(2g/km) + \hat{\phi}(2 m_k g_k/lm)}{2} \le \delta.
\end{align*}
This shows that $\phi(f + f_n/m) \downarrow \phi(f)$, and the proof is complete.
\qed

\bigskip
\noindent
{\bf Proof of Corollary \ref{corAB}}\\
It follows from Theorem \ref{thmAB} that there exists a $\mu \in ca^+(X)$ such that
$\phi^*_X(\mu) < + \infty$. If $\phi$ is linear, this implies that $\ang{f,\mu} = \phi(f)$ for all $f \in X$.
\qed

\bigskip
\noindent
{\bf Proof of Proposition \ref{prop:reg}}\\
Fix a finite measure $\mu$ on $\sigma(X)$ and call a set $A \in \sigma(X)$ {\sl closed regular} if 
\[
\mu(A) = \sup \crl{\mu(B) : \mbox{$B \in \sigma(X)$, $B$ is closed and $B \subseteq A$}}.
\]
The collection of sets
\[
{\cal G} := \crl{A \in \sigma(X) : \mbox{ $A$ and $\Omega \setminus A$ are closed regular}}
\]
forms a sub-$\sigma$-algebra of $\sigma(X)$.
For a closed set $F \subseteq \mathbb{R}$ and $f \in X$, $f^{-1}(F)$ is a closed subset of $\Omega$.
Moreover, $\mathbb{R} \setminus F$ can be written as a countable union $\bigcup_n F_n$ of closed sets 
$F_n \subseteq \mathbb{R}$. Therefore, $\Omega \setminus f^{-1}(F)$ equals $\bigcup_n f^{-1}(F_n)$, which 
can be approximated with the closed sets $\bigcup_{n=1}^N f^{-1}(F_n)$. This shows that
$f^{-1}(F)$ belongs to ${\cal G}$. Since the sets $f^{-1}(F)$ generate $\sigma(X)$, one obtains 
${\cal G} = \sigma(X)$, which means that $\mu$ is closed regular.

If there exists a sequence $(K_n)$ of compact sets in $\sigma(X)$ such that $\mu(K_n) \to \mu(\Omega)$, 
then $\mu(A \cap K_n) \to \mu(A)$ for every $A \in {\cal F}$. Moreover, for every $n$ there exists a 
closed set $B_n \subseteq A \cap K_n$ in $\sigma(X)$ such that $\mu(B_n) \ge \mu(A \cap K_n) - 1/n$.
Since every closed subset of a compact set is compact, this shows that $\mu$ is regular.
\qed

\setcounter{equation}{0}
\section{Derivation of sup-representations}
\label{sec:sup}

For a sequence of non-empty compact Hausdorff spaces $(H_n)$, consider the sequence spaces 
\[
U := \crl{u \in \prod_n C(H_n) : \N{u} < \infty} \mbox{ and } 
V := \crl{\nu \in \prod_n ca_r(H_n) : \N{\nu} < \infty},
\]
where $C(H_n)$ denotes the set of all real-valued continuous functions on $H_n$,
$ca_r(H_n) = ca^+_r(H_n) - ca^+_r(H_n)$, where $ca^+_r(H_n)$ are all finite regular measures on the 
Borel $\sigma$-algebra of $H_n$, and the norms are defined as follows:
\beas
&& \N{u} := \sup_n \N{u_n}_{\infty} \mbox{ for the sup-norm $\N{.}_{\infty}$ and }\\
&& \N{\nu} := \sum_n \N{\nu_n}_{\rm tv} < \infty \mbox{ for the total variation norm $\N{.}_{\rm tv}$.}
\eeas
By the Riesz--Markov--Kakutani representation theorem (see, e.g., Theorem IV.6.3 in \cite{DS}), 
$ca_r(H_n)$ is the topological dual of $C(H_n)$. Therefore,
$(U,V)$ is a dual pair under the bilinear form $\ang{u,\nu} := \sum_n \ang{u_n,\nu_n}$.
By $V^+$ we denote the set of all $\nu \in V$ belonging to $\prod_n ca^+_r(H_n)$.
For a function $\psi \colon U \to \mathbb{R} \cup \crl{+ \infty}$, we consider the following two conditions:

\begin{itemize}
\item[{\bf (C')}]
$\psi$ is real-valued and $\psi(u^n) \uparrow \psi(u)$ 
for every increasing sequence $(u^n)$ in $U$ and $u \in U$ such that 
$u^n_m = u_m$ for all $n \ge m$.
\item[{\bf (D')}]
$\psi(u^n) \uparrow \psi(u)$ for every increasing sequence $(u^n)$ in $U$ and $u \in U$ such that 
$\lim_{n \to + \infty} \N{u^n_m - u_m}_{\infty} = 0$ for every $m$.
\end{itemize}

Note that $U$ contains $l^{\infty}$ as a subspace, and on $l^{\infty}$ the following holds:

\begin{lemma} \label{lemlinfty}
Every increasing convex functional $\psi \colon l^{\infty} \to \mathbb{R} \cup \crl{+\infty}$ 
satisfying {\rm (C')} or {\rm (D')} is $\sigma(l^{\infty}, l^1)$-lower semicontinuous. 
\end{lemma}

\begin{proof}
The lemma follows if we can show that all lower level sets of $\psi$ are
$\sigma(l^{\infty}, l^1)$-closed. By the Krein--\v{S}mulian theorem
(see, e.g., Theorem V.5.7 in \cite{DS}), it is enough to show that the sets 
\[
D_{a,b} = \crl{x \in l^{\infty} : \psi(x) \le a, \; \N{x}_{\infty} \le b}, \quad a,b \in \mathbb{R},
\]
are $\sigma(l^{\infty}, l^1)$-closed, which is the case if and only if they are 
$\sigma(l^{\infty}, l^1(\eta))$-closed, where $l^1(\eta)$ is the 
$l^1$-space with respect to the probability measure $\eta$ on 
$\mathbb{N}$ given by $\eta(n) = 2^{-n}$, and the pairing 
on $(l^{\infty}, l^1(\eta))$ is $\ang{x,y} = \sum_n x_n y_n 2^{-n}$.
The embedding of $l^{\infty}$ in $l^1(\eta)$ is continuous with respect to 
$\sigma(l^{\infty}, l^1(\eta))$ and $\sigma(l^1(\eta),l^{\infty})$.
So it is sufficient to show that the sets $D_{a,b}$ are $\sigma(l^1(\eta),l^{\infty})$-closed.
But by convexity, this follows if it can be shown that they are norm-closed in $l^1(\eta)$.
To do that, consider a sequence $(x^n)$ in $D_{a,b}$ converging to 
$x$ in the $l^1(\eta)$-norm. Then $\N{x}_{\infty} \le b$, and
$y^n_m := \inf_{j \ge n} x^j_m$ defines a sequence $(y^n)$ in $D_{a,b}$
which increases component-wise to $x$. 

Under (C') $\psi$ is real-valued, and since every real-valued convex function on 
$\mathbb{R}^m$ is continuous, one has
\begin{align*}
\psi(x_1, \dots, x_m, y^1_{m+1}, y^1_{m+2}, \dots )
&= \lim_{n \to + \infty} \psi(y^n_1, \dots, y^n_m, y^1_{m+1}, y^1_{m+2}, \dots )\\
& \le \lim_{n \to + \infty} \psi(y^n) \le a
\end{align*}
for all $m \ge 1$. Therefore, it follows from (C') that $x$ belongs to $D_{a,b}$.
If $\psi$ satisfies (D'), one obtains $\psi(x) = \lim_{n \to + \infty} \psi(y^n) \le a$. So $x$ belongs to $D_{a,b}$.
\end{proof}

We also need the following 

\begin{lemma} \label{lemcompact}
For every $y \in l^1$, the set $A_y = \crl{\nu \in V : \N{\nu_n}_{\rm tv} \le |y_n|}$
is $\sigma(V,U)$-compact.
\end{lemma}

\begin{proof}
\[
\tilde{U} := \crl{u \in \prod_n C(H_n) : \sum_n \N{u_n}_{\infty} < \infty}
\]
is a Banach space with topological dual 
\[
\tilde{V} := \crl{\nu \in \prod_n ca_r(H_n) : \sup_n \N{\nu_n}_{\rm tv} < \infty}.
\]
Therefore, one obtains from the Banach--Alaoglu theorem that the norm ball
$\{\nu \in \tilde{V} : \sup_n \N{\nu_n}_{\rm tv} \le 1\}$ is 
$\sigma(\tilde{V}, \tilde{U})$-compact. But for $y \in l^1$, 
the mapping $(\nu_n) \mapsto (\nu_n y_n)$ continuously embeds 
$\tilde{V}$ in $V$ with respect to $\sigma(\tilde{V}, \tilde{U})$ and 
$\sigma(V,U)$. It follows that $A_y$ is $\sigma(V,U)$-compact.
\end{proof}

Now we are ready to prove a representation result for increasing convex functionals on $U$.

\begin{proposition} \label{UVlsc} 
Every increasing convex functional $\psi \colon U \to \mathbb{R} \cup \crl{+ \infty}$ satisfying
{\rm (C')} or {\rm (D')} has a representation of the form
\[
\psi(u) = \sup_{\nu \in V^+} \brak{\ang{u,\nu} - \psi^*(\nu)}
\quad \mbox{for} \quad \psi^*(\nu) := \sup_{u \in U} \brak{\ang{u, \nu} - \psi(u)}.
\]
\end{proposition}

\begin{proof}
In the case $\psi \equiv + \infty$, the proposition is clear. So let us assume that $\psi(u) < + \infty$
for at least one $u \in U$.
Then it is enough to show that
\be \label{psirep}
\psi(u) = \sup_{\nu \in V} \brak{\ang{u,\nu} - \psi^*(\nu)}, \quad u \in U,
\ee
since it follows from the monotonicity of $\psi$ that $\psi^*(\nu) = + \infty$ for all $\nu \in V \setminus V^+$.
But \eqref{psirep} is a consequence of the Fenchel--Moreau theorem 
(see, e.g., Theorem 3.2.2 in \cite{Zalinescu}) if we can show that 
$\psi$ is $\sigma(U,V)$-lower semicontinuous, or equivalently, all lower level sets 
$D_a = \crl{u \in U : \psi(u) \le a}$ are $\sigma(U,V)$-closed. 
Moreover, since every $D_a$ is convex, it follows from the Hahn--Banach separation theorem 
together with the Mackey--Arens theorem (see, e.g., Theorem IV.3.2 in \cite{SW}) that it is 
$\sigma(U,V)$-closed if we can show that it is closed in the Mackey topology 
$\tau(U,V)$. So let $(u^{\alpha})$ be a net in $D_a$ such that $u^{\alpha} \to \hat{u} \in U$ in $\tau(U,V)$. 
We know from Lemma \ref{lemcompact} that for every $y \in l^1$, the set
$A_y := \crl{\nu \in V : \N{\nu_n}_{\rm tv} \le |y_n|}$
is $\sigma(V,U)$-compact. Therefore, one has
\be \label{uaconv}
\sum_n \N{u^{\alpha}_n - \hat{u}_n}_{\infty} |y_n|
\le \sup_{\nu \in A_y} \abs{\ang{u^{\alpha} - \hat{u},\nu}} \to 0.
\ee
If $\psi$ satisfies (C'), we define
the projections $\pi_n \colon U \to l^{\infty}$ as follows: for $m > n$,
\[
\pi_n(u)_m := \underline{u}_m := \min_{z \in H_m} u_m(z),
\]
and for $m =1, \dots, n$,
\begin{align*}
\pi_n(u)_1 &:= \min \big\{x \in \mathbb{R} : x \ge \underline{u}_1, \, \psi(x, u_2, \dots, u_n, \underline{u}_{n+1}, \dots)
\\
&\qquad\qquad \qquad \qquad \qquad\qquad\qquad\qquad = \psi(u_1, u_2, \dots, u_n, \underline{u}_{n+1}, \dots)\big\}\\
\pi_n(u)_2 &:= \min \big\{x \in \mathbb{R} : x \ge \underline{u}_2, \, 
 \psi(\pi_n(u)_1, x, u_3, \dots, u_n, \underline{u}_{n+1}, \dots)\\
&\qquad\qquad \qquad \qquad \qquad\qquad\qquad \qquad = \psi(u_1, \dots, u_n, \underline{u}_{n+1}, \dots)\big\}\\
 & \dots\\
 \pi_n(u)_n &:= \min \big\{x \in \mathbb{R} : x \ge \underline{u}_n, \, 
 \psi(\pi_n(u)_1, \dots, \pi_n(u)_{n-1}, x, \underline{u}_{n+1}, \dots)\\ 
&\qquad\qquad \qquad \qquad \qquad\qquad\qquad\qquad = \psi(u_1, \dots, u_n, \underline{u}_{n+1}, \dots) \big\}.
\end{align*}
Since $x \mapsto \psi(x, u_2, \dots, u_n, \underline{u}_{n+1}, \dots)$ is a convex function 
from $\mathbb{R}$ to $\mathbb{R}$, it is continuous. Therefore, the minimum in the 
definition of $\pi_n(u)_1$ is attained, and 
\[
\psi(\pi_n(u)_1, u_2, \dots, u_n, \underline{u}_{n+1}, \dots)=
\psi(u_1, u_2, \dots, u_n, \underline{u}_{n+1}, \dots).
\]
Analogously, the other minima are attained, and 
\[
\psi \circ \pi_n (u) = \psi(u_1, \dots, u_n, \underline{u}_{n+1}, \dots) \quad \mbox{for all } u \in U.
\]
Since $\psi$ is increasing, $\pi_n(u^{\alpha})$ is in $D_a$ for all $\alpha$, and
by \eqref{uaconv}, one has for each $y \in l^1$,
\[
\abs{\ang{\pi_n(u^{\alpha}) -\pi_n(\hat{u}), y}} \le 
\sum_m \N{u^{\alpha}_m - \hat{u}_m}_{\infty} |y_m| \to 0,
\]
showing that $\pi_n(u^{\alpha}) \to \pi_n(\hat{u})$ in $\sigma(l^{\infty}, l^1)$.
From Lemma \ref{lemlinfty} we know that $\psi$ restricted to $l^{\infty}$ is $\sigma(l^{\infty},l^1)$-lower semicontinuous.
Therefore, $\pi_n(\hat{u})$ is in $D_a$ for all $n$, and one obtains from (C') that
\[
\psi \circ \pi_n (\hat{u}) = \psi(\hat{u}_1, \dots, \hat{u}_n, \underline{\hat{u}}_{n+1}, \dots) 
\uparrow \psi(\hat{u}) \quad \mbox{for } n \to +\infty.
\]
This shows that $\hat{u}$ belongs to $D_a$, which completes the proof in the case
where $\psi$ satisfies (C').

If $\psi$ fulfills (D'), we fix $n \ge 1$ and note that due to \eqref{uaconv}, there exists an $\alpha_0$ 
such that 
\[
u^{\alpha}_m \ge \hat{u}_m - \frac{1}{n} \quad \mbox{for all } \alpha \ge \alpha_0
\mbox{ and } m =1, \dots, n.
\]
It follows that
\[
\brak{\hat{u}_1 - \frac{1}{n}, \dots, \hat{u}_n - \frac{1}{n}, \underline{u}^{\alpha}_{n+1} - \frac{1}{n}, 
\underline{u}^{\alpha}_{n+2} - \frac{1}{n}, \dots}
\mbox{ is in } D_a \mbox{ for all }\alpha \ge \alpha_0.
\]
As above, one deduces from \eqref{uaconv} that
\[
\brak{\underline{u}^{\alpha}_{n+1} - \frac{1}{n}, \underline{u}^{\alpha}_{n+2} - \frac{1}{n}, \dots} \to 
\brak{\underline{\hat{u}}_{n+1} - \frac{1}{n}, \underline{\hat{u}}_{n+2} - \frac{1}{n}, \dots} \mbox{ in } \sigma(l^{\infty}, l^1).
\]
So, since 
\[
x \mapsto \psi \brak{\hat{u}_1 - \frac{1}{n}, \dots, \hat{u}_n - \frac{1}{n}, x_1, x_2, \dots}
\]
defines an increasing convex mapping on $l^{\infty}$ with property (D'), one obtains from 
Lemma \ref{lemlinfty} that
\[
\brak{\hat{u}_1 - \frac{1}{n}, \dots, \hat{u}_n - \frac{1}{n}, \underline{\hat{u}}_{n+1} - \frac{1}{n},
\underline{\hat{u}}_{n+2} - \frac{1}{n}, \dots} \mbox{ belongs to } D_a \mbox{ for all } n \ge 1.
\]
Now it follows from (D') that $\hat{u}$ is in $D_a$, and the proof is complete.
\end{proof}

\bigskip
\noindent
{\bf Proof of Theorem \ref{thm:Borel}}\\
We first prove \eqref{repcont}. It is immediate from the definition of $\phi^*_{C_b}$ that 
$\phi(f) \ge \sup_{\mu \in ca^+_r} (\ang{f,\mu} - \phi^*_{C_b}(\mu))$
for all $f \in C_b$. We show the other inequality in the following three steps:

Step 1: For $H_n = K_n$, define the function $\psi \colon U = \prod_n C(H_n) \to \mathbb{R} \cup \crl{+ \infty}$ by
\[
\psi(u) := \phi \brak{\sum_n u_n 1_{K_n \setminus K_{n-1}}}, \quad \mbox{where } K_0 := \emptyset.
\]
Then $\psi$ is increasing and convex. Moreover, it fulfills (C') or (D') if $\phi$ satisfies (C) or (D), 
respectively. So it follows from Proposition \ref{UVlsc} that 
\[
\psi(u) = \sup_{\nu \in V^+} \brak{\ang{u,\nu} - \psi^*(\nu)} \quad \mbox{for all } u \in U.
\]

Step 2:
For every $\nu \in V^+$, $\mu_{\nu}(A) = \sum_n \nu_n(A \cap K_n)$ defines an element of $ca^+_r({\cal F})$. 
Indeed, $\mu_{\nu}$ is a finite measure since $\N{\nu} = \sum_n \N{\nu_n}_{\rm tv} < \infty$.
Moreover, for given $A \in {\cal F}$ and $\varepsilon > 0$, there exist compact sets $F_n \subseteq A \cap K_n$
such that $\nu_n(F_n) \ge \nu_n(A \cap K_n) - 2^{-n-1} \varepsilon$. So, for $m \in \mathbb{N}$
large enough, $F = \bigcup_{n = 1}^m F_n$ is compact, 
$F \subseteq A$ and $\mu_{\nu}(F) \ge \mu_{\nu}(A) - \varepsilon$.

Step 3: Since $\phi$ satisfies (C) or (D), one has for each $f \in C_b$,
\[
\phi(f) = \phi(f 1_{\bigcup_n K_n}) = \psi(f \mid_{K_1}, f \mid_{K_2}, ...).
\]
Therefore,
\[
\phi(f) = \sup_{\nu \in V^+} \brak{\sum_n \ang{f \mid_{K_n}, \nu_n} - \psi^*(\nu)}
= \sup_{\nu \in V^+} \brak{\ang{f, \mu_{\nu}} - \psi^*(\nu)}, 
\]
and it is enough to show that $\phi^*_{C_b}(\mu_{\nu}) \le \psi^*(\nu)$ for all $\nu \in V^+$ 
to complete the proof of \eqref{repcont}. But this readily follows from
\beas
\phi^*_{C_b}(\mu_{\nu}) &=& \sup_{f \in C_b} \brak{\ang{f, \mu_{\nu}} - \phi(f)}
= \sup_{f \in C_b} \brak{\sum_n \ang{f \mid_{K_n}, \nu_n} - \psi(f \mid_{K_1}, f \mid_{K_2}, ...)}\\
&\le&  \sup_{u \in U} \brak{\ang{u, \nu} - \psi(u)} = \psi^*(\nu).
\eeas

To show \eqref{inequsc} we fix an $f \in U_b$ and a constant $\varepsilon > 0$. 
For every $\delta > 0$, there exists a measurable partition $(A_m)_{m =1}^M$ of $\Omega$ 
and real numbers $a_1 < \dots < a_M$ such that the step function $g = \sum_{m=1}^M a_m 1_{A_m}$ 
satisfies $g \le f \le g + \delta$, and by passing to the upper semicontinuous hull, 
one can assume $g$ to be upper semicontinuous. If $\phi$ satisfies (C), then 
$x \mapsto \phi(f + x)$ defines a convex function from $\mathbb{R}$ to $\mathbb{R}$. 
So it has to be continuous, and since $\phi$ is increasing, one can ensure that
$\phi(g) \ge \phi(f) -\varepsilon$ by choosing $\delta >0$ small enough. 
If $\phi$ satisfies (D) and $\phi(f) < + \infty$, one obtains directly that 
$\phi(g) \ge \phi(f) - \varepsilon$ for $\delta > 0$ small enough. On the other hand, if $\phi$ satisfies 
(D) and $\phi(f) = + \infty$, then $\phi(g) \ge \varepsilon$ for $\delta > 0$ small enough.
Now denote
\begin{align*}
U^M &:= \crl{u \in \prod_n C(K_n)^M : \sup_{n,m} \N{u_{nm}}_{\infty} < \infty},\\ \quad
V^M &:= \crl{\nu \in \prod_n ca_r(K_n)^M : \sum_{n,m} \N{\nu_{nm}}_{\rm tv} < \infty},
\end{align*}
and define $\psi \colon U^M \to \mathbb{R} \cup \crl{+ \infty}$ by
\[
\psi(u) := \phi \brak{\sum_n \sum_{m=1}^M u_{nm} 1_{B_{nm}}},
\]
where $K_0 := \emptyset$ and $B_{nm} := (K_n \setminus K_{n-1}) \cap A_m$.
Then $\psi$ is increasing, convex and satisfies (C') or (D'). Therefore, it follows from 
Proposition \ref{UVlsc} that
\[
\psi(u) = \sup_{\nu \in (V^M)^+} \brak{\ang{u,\nu} - \psi^*(\nu)}, \mbox{ where} \quad
\psi^*(\nu) = \sup_{u \in U^M} \brak{\ang{u,\nu} - \psi(u)}.
\]
If $\phi(h) = + \infty$ for all $h \in C_b$, then $\phi^*_{C_b} \equiv - \infty$, and 
\eqref{inequsc} is obvious. So let us assume there exists an $h \in C_b$ such that $\phi(h) < + \infty$.
Then it follows that $\nu_{nm}(K_n \setminus \bar{B}_{nm}) = 0$
for all $\nu \in (V^M)^+$ satisfying $\psi^*(\nu) < + \infty$. Indeed, assume 
$\nu_{nm}(K_n \setminus \bar{B}_{nm}) > 0$. Then, since $\nu_{nm}$ is regular, 
there exists a closed set $F \subseteq K_n \setminus \bar{B}_{nm}$ with $\nu_{nm}(F) > 0$.
By Theorem 2.48 in \cite{AB}, $K_n$ is normal. So it follows from Urysohn's lemma 
that there exists a continuous function $\varphi \colon K_n \to [0,1]$ which is $1$ 
on $F$ and $0$ on $\bar{B}_{nm}$. This gives
\beas
\psi^*(\nu) &\ge& \sup_{x \in \mathbb{R}_+} \brak{ \sum_i \sum_{j=1}^M \ang{h \mid_{K_i}, \nu_{ij}}
+ \ang{x\varphi,\nu_{nm}} - \phi(h + x \varphi 1_{B_{nm}})}\\
&=& \sup_{x \in \mathbb{R}_+} \brak{ \sum_i \sum_{j=1}^M \ang{h \mid_{K_i}, \nu_{ij}}
+ \ang{x\varphi,\nu_{nm}} - \phi(h)} = + \infty.
\eeas
Now define $u \in U^M$ by $u_{nm} = a_m$. Since $\phi$ satisfies (C) or (D), one has
\[
\phi(g) = \phi(g 1_{\bigcup_n K_n}) = \psi(u),
\]
and therefore, by the upper semicontinuity of $g$,
\beas
&& \phi(g) = \sup_{\nu \in (V^N)^+} \brak{\sum_n
\sum_{m=1}^M \ang{u_{nm} , \nu_{nm}} - \psi^*(\nu)}\\ &\le& \sup_{\nu \in (V^N)^+} \brak{\sum_n \sum_{m=1}^M
\ang{g 1_{\bar{B}_{nm}} , \nu_{nm}} - \psi^*(\nu)} = \sup_{\nu \in (V^N)^+} \brak{\ang{g, \mu_{\nu}} - \psi^*(\nu)},
\eeas
where $\mu_{\nu}$ is given by $\mu_{\nu}(A) = \sum_n \sum_{m=1}^M \nu_{nm}(A \cap K_n)$. 
It follows as above that $\mu_{\nu}$ belongs to $ca^+_r({\cal F})$, and 
for all $\nu \in (V^M)^+$, one has
\beas
&& \phi^*_{C_b}(\mu_{\nu}) = \sup_{l \in C_b} \brak{\ang{l,\mu_{\nu}} - \phi(l)}\\
&=& \sup_{l \in C_b} \brak{\sum_n \sum_{m=1}^M \ang{l \mid_{K_n},\nu_{nm}} 
- \psi(l \mid_{K_1}, \dots, l \mid_{K_1}, l \mid_{K_2}, \dots)}\\
&\le& \sup_{u \in U^M} \brak{\ang{u,\nu} - \psi(u)} = \psi^*(\nu).
\eeas
So in the case $\phi(f) < + \infty$, one obtains 
\[
\phi(f) - \varepsilon \le \phi(g) \le \sup_{\mu \in ca^+_r} \brak{\ang{g,\mu} - \phi^*_{C_b}(\mu)}
\le \sup_{\mu \in ca^+_r} \brak{\ang{f,\mu} - \phi^*_{C_b}(\mu)},
\]
and if $\phi(f) = + \infty$, 
\[
\varepsilon \le \phi(g) \le \sup_{\mu \in ca^+_r} \brak{\ang{g,\mu} - \phi^*_{C_b}(\mu)}
\le \sup_{\mu \in ca^+_r} \brak{\ang{f,\mu} - \phi^*_{C_b}(\mu)}.
\]
Since $\varepsilon > 0$ was arbitrary, this yields \eqref{inequsc}. On the other hand, 
it follows from the definition of $\phi^*_{U_b}$ that $\phi(f) \ge \sup_{\mu \in ca^+_r} 
(\ang{f,\mu} -\phi^*_{U_b}(\mu))$. So if $\phi^*_{C_b}(\mu) = \phi^*_{U_b}(\mu)$ 
for all $\mu \in ca^+_r$, the inequality in \eqref{inequsc} becomes an equality.

Finally, by Remark \ref{remlowerreg}.1, $\hat{\phi}(f) = \sup_{\mu \in ca^+_r} (\ang{f,\mu} - \phi^*_{C_b}(\mu))$
is lower regular on $B_b$. So one obtains from the second part of the proof that for all $f \in B_b$,
\[
\phi_r(f) = \sup\{ \phi(g): g \in U_b, g \le f \} \le \sup\{\hat{\phi}(g): g \in U_b, g \le f \} = \hat{\phi}(f),
\]
with equality if $\phi^*_{C_b}(\mu) = \phi^*_{U_b}(\mu)$ for all $\mu \in ca^+_r$.
This completes the proof.
\qed

\bigskip
\noindent
{\bf Proof of Corollary \ref{cor:lin}}\\
By Theorem \ref{thm:Borel}, one has 
\[
\phi(f) = \sup_{\mu \in ca^+_r({\cal F})} (\ang{f,\mu} - \phi^*_{C_b}(\mu))
\quad \mbox{for all } f \in C_b.
\]
In particular, $\phi^*_{C_b}(\mu) < + \infty$ for at least one $\mu \in ca^+_r({\cal F})$. 
Since $\phi$ is linear, this implies that $\phi(f) = \ang{f,\mu}$ for all $f \in C_b$ 
and $\phi^*_{C_b}(\mu) = 0$. Moreover, if $\Omega$ is a metric space, $\mu$ is completely 
determined by the values $\ang{f, \mu}$, $f \in C_b$; see, e.g., \cite{Bil}. So one obtains from 
\eqref{inequsc} and \eqref{ineqBorel} that $\phi(f) \le \ang{f,\mu}$ for all $f \in U_b$
and $\phi_r(f) \le \ang{f,\mu}$ for all $f \in B_b$, with equality if 
$\phi^*_{C_b}(\nu) = \phi^*_{U_b}(\nu)$ for all $\nu \in ca^+_r({\cal F})$.
\qed


\begin{thebibliography}{20}
\bibitem{AB}
C.D. Aliprantis and K.C. Border (2006). Infinite Dimensional Analysis. 
A Hitchhiker's Guide. 3rd Ed. Springer-Verlag.

\bibitem{BCK}
D. Bartl, P. Cheridito and M. Kupper (2019).
Robust expected utility maximization with medial limits.
Journal of Mathematical Analysis and Applications,
471(1-2) 752--775.

\bibitem{BDT}
D. Bartl, S. Drapeau and L. Tangpi (2020).
Computational aspects of robust optimized certainty equivalent and option pricing. 
Mathematical Finance 30(1) 287--309.

\bibitem{BF}
S. Biagini and M. Frittelli (2009).
On the extension of the Namioka--Klee theorem and on the Fatou property for risk measures. 
Optimal and Risk-Modern Trends in Mathematical Finance, F. Delbaen, M. Rasonyi and Ch. 
Stricker editors. Springer.

\bibitem{Bil}
P. Billingsley (1999).
Convergence of Probability Measures. 2nd Ed. Wiley Series in Probability and Statistics.
John Wiley \& Sons, Inc.

\bibitem{Bog}
V. Bogachev (2007).
Measure Theory. Volume II. Springer.

\bibitem{Bor}
J.M. Borwein (1987). Automatic continuity and openness of convex relations. 
Proceedings of the American Mathematical Society 99(1), 49--55.

\bibitem{CMMM}
S. Cerreia-Vioglio, F. Maccheroni, M. Marinacci and L. Montrucchio (2011).
J. Econ. Theory 146(4), 1275--1330.

\bibitem{CKPS}
P. Cheridito, M. Kiiski, D.J. Prömel and H.M. Soner (2020).
Martingale optimal transport duality.
Math. Ann.

\bibitem{CKT}
P. Cheridito, M. Kupper and L. Tangpi (2017).    
Duality formulas for robust pricing and hedging in discrete time. 
SIAM Journal of Financial Mathematics 8(1), 738--765.

\bibitem{CL}
P. Cheridito and T. Li (2009).
Risk measures on Orlicz hearts.
Math. Finance 19(2), 189--214.

\bibitem{C}
G. Choquet (1962).
Le probl\`eme des moments. S\'eminaire Choquet.
Initiation à l'Analise 1(4), 1--10.

\bibitem{Daniell}
P.J. Daniell (1918). A general form of integral. Ann. Math. 19(4), 279--294.

\bibitem{Del2000}
F. Delbaen (2000).
Coherent Risk Measures. Cattedra Galileiana. Scuola Normale Superiore,
Classe di Scienze, Pisa.

\bibitem{Del2002}
F. Delbaen (2002). Coherent risk measures on general probability spaces. In: Advances in
Finance and Stochastics. Essays in Honour of Dieter Sondermann, Springer-Verlag,
Berlin, 1--37.

\bibitem{DS94}
F. Delbaen and W. Schachermayer (1994). A general version of the fundamental theorem of asset
pricing. Math. Ann. 300 , 463--520.

\bibitem{DS98}
F. Delbaen and W. Schachermayer (1998). The fundamental theorem of asset pricing 
for unbounded stochastic processes. Math. Ann. 312, 215--250.

\bibitem{DK}
S. Drapeau and M. Kupper (2013).
Risk preferences and their robust representation.
Math. Oper. Research 38(1), 28--62.

\bibitem{Dudley} R.M. Dudley (2004).
Real Analysis and Probability. Cambridge Studies in Advanced Mathematics 74.
Cambridge University Press.

\bibitem{DS}
N. Dunford and J.T. Schwartz (1958). Linear Operators. Part I: General Theory. Interscience Publishers,
New York, 1958.

\bibitem{FS02a}
H. Föllmer and A. Schied (2002).
Convex measures of risk and trading constraints. Finance \& Stochastics 6(4), 429--447.

\bibitem{FS02b}
H. Föllmer and A. Schied (2002). Robust representation of convex measures of risk. In: Advances
in Finance and Stochastics. Essays in Honour of Dieter Sondermann, Springer-Verlag,
Berlin, 39–56.

\bibitem{FS16}
H. F\"ollmer and A. Schied (2016).
Stochastic Finance: An Introduction in Discrete Time. 4th Edition. 
De Gruyter, Berlin/Boston.

\bibitem{Kr}
V. Kr\"atschmer (2005).
Robust representation of convex risk measures by probability measures.
Finance \& Stochastics 9(4), 597–608 (2005)

\bibitem{MMR}
F. Maccheroni, M. Marinacci and A. Rustichini (2006). Ambiguity aversion, robustness, and the
variational representation of preferences. Econometrica 74, 1447--1498.

\bibitem{MS}
G. Mokobozky and D. Sibony (1966/67). 
C\^{o}nes adapt\'{e}s de fonctions continues et th\'{e}orie du potentiel.
S\'eminaire Choquet. Initation \`{a} l'Analyse 6(5), 1--35.

\bibitem{RS}
A. Ruszczyński and A. Shapiro (2006).
Optimization of convex risk functions.
Math. Oper. Research 31(3), 433--452.

\bibitem{SW}
H.H. Schaefer and M.P. Wolff (1999).
Topological Vector Spaces. 2nd ed. Graduate Texts in Mathematics.
Springer-Verlag New York.

\bibitem{Zalinescu}
C. Z\u{a}linescu (2002). Convex Analysis in General Vector Spaces. World Scientific.

\end{thebibliography}
\end{document}